\documentclass[reqno]{amsart}
\usepackage{}
\usepackage[top=1in, bottom=1in, left=1in, right=1in]{geometry}
\usepackage{cases}
\usepackage[colorlinks,
            linkcolor=red,
            anchorcolor=blue,
            citecolor=green
            ]{hyperref}
\usepackage{amsmath,amsthm,amsfonts,amssymb,amscd,bm}
\usepackage{url}
\usepackage{longtable}
\usepackage{ragged2e}
\usepackage{tabularx}
\usepackage{array}
\usepackage{supertabular}
\newtheorem{theorem}{Theorem}[section]
\newtheorem{lemma}[theorem]{Lemma}
\theoremstyle{definition}

\newtheorem{example}[theorem]{Example}

\newtheorem{cor}[theorem]{Corollary}

\newtheorem{alg}[theorem]{Algorithm}
\theoremstyle{remark}
\newtheorem{remark}[theorem]{Remark}
\renewcommand{\dim}{\mbox{\rm dim}}

\newcommand{\GL}{\mbox{\rm GL}}

\newcommand{\SL}{\mbox{\rm SL}}

\newcommand{\Gal}{\mbox{\rm Gal}}
\newcommand{\MOD}{\mbox{\rm ~mod~}}
\newcommand{\GQ}{\Gal(\overline{\mathbb{Q}}/\mathbb{Q})}
\newcommand{\Frob}{\mbox{\rm Frob}}

\numberwithin{equation}{section}

\begin{document}\large

\title{Computing Galois representations and equations for modular curves $X_H(\ell)$}

\author{Maarten Derickx, Mark van Hoeij, Jinxiang Zeng}
\address{Mathematisch Instituut, Universiteit Leiden, Postbus 9512, 2300 RA Leiden, The Netherlands}
\email{mderickx@math.leidenuniv.nl}
\address{Department of Mathematics, Florida State University, Tallahassee, Florida 32306, USA}
\email{hoeij@math.fsu.edu}
\address{Department of Mathematical Science, Tsinghua University, Beijing 100084, P. R. China}
\email{cengjx09@mails.tsinghua.edu.cn}
\thanks{Second author supported by NSF grant 1319547}






\begin{abstract}
We construct plane models of the modular curve $X_H(\ell)$ and describe the moduli interpretation on these plane models. We use these explicit plane models to
compute Galois representations associated to modular forms for values of $\ell$ that are
significantly higher than in prior works.
\end{abstract}

\maketitle

\section{Introduction}

Couveignes, Edixhoven et al. \cite{Edixhoven} described polynomial time algorithms for computing Galois representations
associated to modular forms for the group SL$_2(\mathbb{Z})$. Bruin \cite{Bruin} generalized the method to modular forms for
congruence subgroups of the form $\Gamma_1(n)$. A direct consequence of their results is that Fourier coefficients of modular forms can be computed in polynomial time. As a typical example, the value of Ramanujan's $\tau$-fuction
at a prime $p$ can be computed in time bounded by a polynomial in $\log p$.

Progress has been made in designing and implementing practical variants of these algorithms.
A numerical approximation method was first implemented by Bosman \cite{Bosman} and improved
by Mascot \cite{Mascot} and Tian \cite{Tian}.
An algebraic method was first implemented by Zeng \cite{ZengYin}. In this approach, the representation is computed modulo numerous prime numbers $p$,
and then reconstructed with the Chinese Remainder theorem.
Previously the modular Galois representation associated to $\tau(p)\mod\ell$ has already been computed for $\ell\in\{11,13,17,19,29,31\}$.
In the numerical as well as the algebraic method, the main task is to
construct a certain subspace $V_\ell$ of $J_1(\ell)[\ell]$.

For each $k\in\{12,16,18,20,22,26\}$, let $\Delta_{k}=\sum_{n\ge1}\tau_k(n)q^n\in S_k(\SL_2(\mathbb{Z}))$ be the unique newform in $S_k(\SL_2(\mathbb{Z}))$
($\Delta_{12}$ is called the discriminant modular form).
For each prime number $\ell>k$, associated to $\Delta_k$ there is a continuous representation,
$$\rho_{k,\ell}:\GQ\to\textrm{Aut}(V_{k,\ell})\cong\GL_2(\mathbb{F}_\ell)$$
where $V_{k,\ell}=J_1(\ell)[\mathfrak{m}_{k,\ell}]$ is a two-dimensional $\mathbb{F}_\ell$-vector space. Here $\mathfrak{m}_{k,\ell}$ is the maximal ideal of the Hecke algebra $\mathbb{T}=\mathbb{Z}[T_n:n\ge1]\subset\textrm{End}(J_1(\ell))$, generated by $\ell$ and $T_n-\tau_k(n)$ for all $n\ge1$. The representation $\rho_{{k},\ell}$ has the following properties: it is unramified at each prime number $p$ not equal to $\ell$, and
$$\textrm{Tr}(\rho_{{k},\ell}(\Frob_p))\equiv \tau_k(p)\mod\ell,$$
$$\textrm{Det}(\rho_{{k},\ell}(\Frob_p))\equiv p^{k-1}\mod\ell.$$
So to compute $\tau_k(p)\mod\ell$, it suffices to compute $\rho_{{k},\ell}$, and this comes down to computing the representation space $V_{k,\ell}$ explicitly.

Let $\Delta_{k,\ell}\in S_2(\Gamma_1(\ell))$ be the newform with Dirichlet character $\chi$, which is congruent to $\Delta_{k}$ modulo $\ell$. Then
$\chi(p)\equiv p^{k-2}\mod \ell$ for all prime numbers $p\not=\ell$.
Let $d=\gcd(k-2,\ell-1)$, then $\chi(p^{\frac{\ell-1}{d}})\equiv(p^{k-2})^{\frac{\ell-1}{d}}\equiv1\mod\ell$.
In other words, $\chi$ is trivial on the subgroup $H$ of $G=(\mathbb{Z}/\ell\mathbb{Z})^\times$ of
order $|H| = d$.
Let $\Gamma_H(\ell)$ be the congruence subgroup of $\SL_2(\mathbb{Z})$ defined as
$$\Gamma_H(\ell)=\left\{ \left[\begin{matrix} a & b\\ c & d \end{matrix}\right]\in \textrm{SL}_2(\mathbb{Z}):
 a \MOD \ell \in H , c\equiv 0 \MOD \ell  \right\}$$
and $\mathcal{H}$ the upper half complex plane. Let $X_H(\ell)$ be the modular curve defined as
$$X_H(\ell)=\Gamma_H(\ell)\backslash \mathcal{H}\cup \mathbb{P}^1(\mathbb{Q})$$
and $J_H(\ell)$ the Jacobian of $X_H(\ell)$. Then define
$V'_{k,\ell}=J_H(\ell)[\mathfrak{m}'_{k,\ell}]$, here $\mathfrak{m}'_{k,\ell}$ is generated by the same
Hecke operators as $\mathfrak{m}_{k,\ell}$ but with the $T_n - \tau_k(n)$ viewed as elements of the Hecke algebra for $\Gamma_H(\ell)$. Now $V'_{k,\ell}$ and $V_{k,\ell}$ will be isomorphic Galois representations so instead of working with $J_1(\ell)$, we compute in $J_H(\ell)$
using an explicit polynomial equation for $X_H(\ell)$. This trick makes the computation considerably faster if $d>2$,
especially when $d = k-2$. We list several examples below.
\renewcommand\arraystretch{1.5}
\tabcolsep=6.5pt
\begin{longtable}{|c|cc|cc|ccc|cc|cccc|cccc|}
 \caption{Comparing dimensions of $J_1(\ell)$, $J_H(\ell)$ and $A_{\Delta_{k,\ell}}$.}\label{comparingdims}\\
 \hline
    $k$     &12  & 12   & 16 & 16 &  18  & 18  &18   &20  &20    &22 &22  &22  &22  &26 &26  &26  &26   \\
 \hline
   $\ell$   &31  & 41   & 29 & 43 &  29  & 37  &41   &31  &37    &29 &31  &37  &41  &29 &31  &37  &41 \\
 \hline
   $G$      &3   &6     &2   &3   &2     &2    &6    &3   &2     &2  &3   &2   &6   &2  &3   &2   &6  \\
 \hline
 $H$&$3^3$&$6^4$ &$2^2$&$3^3$&$2^7$&$2^9$&$6^5$&$3^5$&$2^2$&$2^7$&$3^3$&$2^9$&$6^2$&$2^7$&$3^5$&$2^3$&$6^5$\\
 \hline
 $\dim J_1(\ell)$ &26&51  &22 &57 &22 &40 &51&26 &40      &22 &26 &40 &51 &22 &26 &40 &51\\
 \hline
 $\dim J_H(\ell)$ &{6}&{11} &{4}& {9}&{8}&{16}&{11}&{6}&{4}&{8}& {6}& {16}&{5}&{8}&{6}& {4}&{11}\\
 \hline
 $\dim A_{\Delta_{k,\ell}}$&{4}&{6} &{2}& {2}&{6}& {6}&{8}&{4}& {2}&{6}& {4}& {6}&{2}&{6}& {4}& {2}&{8}\\
 \hline
\end{longtable}

Each entry in the third and fourth row is a generator of the
cyclic group $G=(\mathbb{Z}/\ell\mathbb{Z})^\times$ and $H\subset G$. The
modular Abelian variety associated to newform $\Delta_{k,\ell}$ is denoted as $A_{\Delta_{k,\ell}}$.

We computed 9 cases with $\ell$ larger than $31$ (the previous record), including 5 cases with $\ell > 40$.
The smallest case we could not compute is $k=12, \ell=37$ because $d=2$ for that case.

Our computational results can be summarized as:
\begin{theorem}\label{mainthm}Let $\tilde\rho_{{k},\ell}$ be the projective representation associated to $\Delta_k$ mod $\ell$.
For each $(k,\ell)$ in Table \ref{Polysforprojectiverepns} (Section 4),
the fixed field of $\ker(\tilde\rho_{{k},\ell})$ is the splitting field of the polynomial $Q_{k,\ell}^{red}$ in Table~\ref{Polysforprojectiverepns}.
\end{theorem}

Our data is available online \url{http://www.math.fsu.edu/~hoeij/files/XH} (for the readers convenience, the URL
also lists polynomials that were computed in prior works, with references).
The theorem implies that the Galois group of $Q_{k,\ell}^{red}$ is $\textrm{PGL}_2(\mathbb{F}_\ell)$, which
we verified with Magma \cite{Bosma}.
Initially we obtained polynomials $Q_{k,\ell}$ with large coefficients, the superscript ${red}$ indicates
a size-reduced polynomial defining the same number field (polredabs in PARI/GP).
The explicit polynomials for $k=12$ and $\ell \in \{11,13,17,19,29,31,41\}$ allow one to efficiently compute $\tau(p)$ mod $\ell$ for huge values of $p$. Together with the known congruences for $\tau$ modulo powers of 2, 3, 5, 7, 23 and 691 we were able to verify Lehmer's non-vanishing conjecture for $\tau$ further
than before.
\begin{cor}\label{maincor}The non-vanishing of the Ramanujan tau function $\tau(n)$ holds for all $n$ with
$$n<816212624008487344127999\approx8\cdot10^{23}$$
\end{cor}

\begin{remark} In \cite{ZengYin} this was verified for
$$n<982149821766199295999\approx9\cdot10^{20}.$$
\end{remark}

Section~2 will discuss finding plane models of modular curves $X_H(\ell)$.
Section~3 explains our method for constructing the space $V_{k,\ell}$
and our computational results are summarized in Section~4.

\section{Equations for modular curves $X_H(\ell)$}
The modular curve $Y_1(N):=X_1(N) - \{\textrm{cusps}\}$ parameterizes  isomorphism classes of pairs $(E,P)$ where $E$ is an elliptic curve and $P$ is an $N$-torsion point on $E$. Equations for $X_1(N)$  have been extensively studied by many authors. The approach in Reichert \cite{Reichert}, Baaziz \cite{Baaziz} and Sutherland \cite{Sutherland} not only gives an equation for $X_1(N)$ but they also describe the moduli interpretation in terms of the equation, i.e. they describe how pairs $(E,P)$ correspond to solutions of their equations. This moduli interpretation enables one to compute the action of Hecke operators on points of the equation for $X_1(N)$. More precisely, let $E$ be an elliptic curve over $\mathbb{Q}$ and $P$ a point on $E$ of order exactly $N$.
If $N > 3$, then each pair $(E,P)$ can be represented uniquely in Tate normal form:
\begin{equation}\label{Tatenormalform}
E_{b,c}:Y^2+(1-c)XY-bY=X^3-bX^2,\textrm{ with the point } (0,0) \textrm{ of order exactly } N.
\end{equation}
So $b,c$ can be viewed as  functions on $X_1(N)$ and the function field $\mathbb{Q}(X_1(N))$ is generated by $b,c$.
A polynomial relation between $b,c$ denoted by $F_N \in \mathbb{Z}[b,c]$ gives a plane equation for $X_1(N)$. Let
$n\ge1$ be an integer, the Hecke operator $T_n$ acts on a point $(E_{b,c},(0,0))$ as
$$T_n(E_{b,c},(0,0)):=\sum_C(E_{b,c}/C,(0,0)+C)$$
where the sum is taken over all order $n$ subgroups $C \subset E_{b,c}$ such that $C \cap\langle(0,0)\rangle=\{0_{E_{b,c}}\}$. Each pair $(E_{b,c}/C,(0,0)+C)$ can be represented in Tate normal form (\ref{Tatenormalform}) as well. Hence, we can compute the action of Hecke operators on points of $F_N$.

Let $H$ be a subgroup of $(\mathbb{Z}/N\mathbb{Z})^\times$. In \cite{Derickx} Derickx and van Hoeij
give an example for computing an equation for $X_H(N)$.
Here we extend this to an algorithm for finding equations for $X_H(N)$.
We also describe the moduli interpretation for this equation so that we can apply Hecke operators on its points.

As mentioned, $F_N$ (for $N \ge 4$) is a equation for $X_1(N)$.
The size of $F_N$ increases drastically with $N$.
For $N \ge 6$ one can find smaller polynomials for $X_1(N)$ as follows, define (see \cite{Sutherland}):
$$r=\frac{b}{c},~s=\frac{c^2}{b-c},~b=rs(r-1),~c=s(r-1)$$
and for $N\ge 10$, define
$$x=\frac{s-r}{rs-2r+1},~y=\frac{rs-2+1}{s^2-s-r+1},~r=\frac{x^2y-xy+y-1}{x(xy-1)},~s=\frac{xy-y+1}{xy}.$$
Writing an equation for $X_1(N)$ as a polynomial in $r,s$ (if $6 \le N \le 9$) or $x,y$ (if $N \ge 10$) reduces expression sizes.
We denote this polynomial as $f_N$. Then $f_4:=c$, $f_5:=b-c$, $f_6:=s-1$, $f_7:=s-r$, $f_8:=rs-2r+1$, $f_9:=s^2-s-r+1$,
$f_{10}:=x-y+1$, $f_{11}:=x^2y-xy^2+y-1$, etc.
Explicit expressions for $f_{10},\ldots,f_{189}\in\mathbb{Z}[x,y]$ can be downloaded from Sutherland's website.
We also define $f_2 := {b^4}/{\Delta}$ and $f_3 := b$,  where $\Delta:=b^3(16b^2+(1-20c-8c^2)b+c(c-1)^3)$ is the discriminant of~(\ref{Tatenormalform}).
If $1<k<N$, then $f_k$ is a modular unit for $X_1(N)$ (see \cite{Derickx}).

In the rest of this section, $\lfloor N/2\rfloor$ is denoted by $n$. There are $n+1 $ $\GQ$-orbits of cusps on $X_1(N)$, denoted as $C_0,\ldots,C_n$.
We number them in such as a way that the diamond operator $\langle d\rangle$ sends $C_i$ to $C_j$ where $j\equiv \pm d\cdot i\mod N$. Write
\begin{equation}\label{div_f_k}
\textrm{div}(f_k)=\sum_{0\le i\le n}a_{k,i}C_i
\end{equation}
with $a_{k,i} \in \mathbb{Z}$. The vector $\left(a_{k,i}\right)_{0\le i\le n}$ has been computed in \cite{Derickx}
and is available online for $N \le 300$ and $2 \le k \le n+1$.
Let $M$ be the $n\times (n+1)$ matrix $$(a_{k,i})_{2\le k\le n+1,~0\le i\le n},$$
then
$$(\textrm{div}(f_k))^t_{2\le k\le n+1}=M\cdot(C_0,\ldots,C_{n})^t$$
where $(\textrm{div}(f_k))^t$ is the transpose of row vector $(\textrm{div}(f_k))$.
It is conjectured in \cite{Derickx} that $f_2,\ldots,f_{n+1}$ generate
the group of $\mathbb Q$-rational modular units of $X_1(N)$. This conjecture has been verified for $N\le 100$. So if $N\le 100$, then for any diamond operator $\langle d\rangle$, $\langle d\rangle f_k$ can be represented as (up to a constant in $\mathbb{Q}^\times$)
$$\langle d\rangle f_k=\prod_{2\le i\le n+1}f_i^{e_i},$$
where $e_i\in\mathbb{Z}$ can be determined with linear algebra:
Using (\ref{div_f_k}), compute $m_{k,i} \in \mathbb{Z}$ such that
$$\langle d \rangle\textrm{div}(f_k)=\sum_{0\le i\le n}m_{k,i}C_i.$$
So
$$\langle d \rangle\textrm{div}(f_k)=(e_i)_{2\le i\le n+1}\cdot (\textrm{div}(f_k))^t_{2\le k\le n+1},$$
where $(e_i)_{2\le i\le n+1}$ is the unique vector satisfying
\begin{equation}\label{inverse}
(e_i)_{2\le i\le n+1}\cdot M=(m_{k,i})_{0\le i\le n}.
\end{equation}

For $N\le100$, let $\mathcal{F}_N:=\langle f_2,f_3,\ldots,f_{n+1}\rangle$ be the group of $\mathbb Q$-rational modular units of $X_1(N)$. We have an explicit embedding
\begin{equation}\label{embedding}
\varphi:\mathcal{F}_N\to \mathbb{Z}^{n+1},~f_k\mapsto(a_{k,i})_{0\le i\le n},~2\le k\le n+1,
\end{equation}
and $\varphi(\mathcal{F}_N)$ is a submodule of $\mathbb{Z}^{n+1}$ of rank $n$, denote it by $\mathcal{L}_N$. So we have an isomorphism
\begin{equation}\label{isoofmodunits}
\varphi:\mathcal{F}_N\to \mathcal{L}_N.
\end{equation}
The inverse map of $\varphi$ is given by (\ref{inverse}), i.e. let $w\in \mathcal{L}_N$ then $\varphi^{-1}(w)$ is the unique vector $v$ such that $v\cdot M=w$.
Let $\mathcal{L}_{N,H}$ be the submodule of $\mathcal{L}_N$ consisting of elements that
are invariant under the action of the diamond operators in $H$
$$\mathcal{L}_{N,H}=\bigcap_{d\in H}\ker(\langle d\rangle-1,\mathcal{L}_N).$$
The inverse image $\varphi^{-1}(\mathcal{L}_{N,H})$ is the group of modular units for $X_H(N)$, denote it by $\mathcal{F}_{N,H}$.

If $\mathcal{L}_{N,H}$ has rank at least two, then pick two independent vectors $(v_{1,i})_{0\le i\le n}$,
$(v_{2,i})_{0\le i\le n}$ in $\mathcal{L}_{N,H}$.
Let $X(x,y),Y(x,y)$ be the corresponding elements of $\mathcal{F}_{N,H}$, then
$$\textrm{div}(X)=\sum_{0\le i\le n}v_{1,i}\cdot C_i\textrm{ , }\textrm{div}(Y)=\sum_{0\le i\le n}v_{2,i}\cdot C_i$$
and
$$\deg(X)=\frac{1}{2}\sum_{0\le i\le n}|v_{1,i}|\cdot\deg(C_i)\textrm{ , } \deg(Y)=\frac{1}{2}\sum_{0\le i\le n}|v_{2,i}|\cdot\deg(C_i).$$
The degrees of $X$, $Y$  viewed as functions in $X_H(N)$ are $d_1:=\frac{\deg(X)}{h}$, $d_2:=\frac{\deg(Y)}{h}$ respectively, where $h=|H/\{\pm1\}|$ is the index of $X_H(N)$ in $X_1(N)$.
Assuming $X,Y$ generate $X_H(N)$, a
polynomial relation between $X$ and $Y$ gives a plane equation for $X_H(N)$, and this polynomial has degree $d_2$ in $X$ and degree $d_1$ in $Y$. We prefer low degrees, so we use LLL to select $(v_{1,i}),~(v_{2,i})$.
Write the polynomial relation between $X$ and $Y$ as:
\begin{equation}\label{FXY}
\sum_{i\le d_2,~j\le d_1}c_{i,j}\cdot X(x,y)^i\cdot Y(x,y)^j=0\textrm{ in } \mathbb{Q}(X_1(N)).
\end{equation}
Let $x:=x_0$ be an integer and $\alpha$ a root of the polynomial $f_N(x_0,y)=0$.
Then
\begin{equation}\label{FX0Y0}
\sum c_{i,j}\cdot X(x_0,\alpha)^i\cdot Y(x_0,\alpha)^j=0.
\end{equation}
Now reduce each $\alpha^k$, $k\ge \deg d$  (where $d = [\mathbb{Q}(\alpha):\mathbb{Q}]$)
to a linear combination of $\alpha^k$, $0\le k< d$. The coefficients of $\alpha^k$, $0\le k< d$ in (\ref{FX0Y0})
should then be 0, giving $d$ $\mathbb{Q}$-linear equations for the variables $c_{i,j}$.
Picking integer values $x = x_0,x_1,x_2,\ldots$ produces as many equations as needed; we keep adding equations for the $c_{i,j}$ until
the solution space has dimension 1.

We now summarize the above discussion into an algorithm. To implement Step~1,
first download the divisors \cite{Derickx} of the generators $f_2,f_3,\ldots$ of $\mathcal{F}_N$.
\begin{alg}\label{equtaionforXH}Find an equation for $X_H(N)$.

Input: An integer $N\ge4$ and a subgroup $H$ of $(\mathbb{Z}/N\mathbb{Z})^\times$.

Output: A polynomial defining $X_H(N)$, or ``no output''.
\begin{itemize}
\item[1.] Compute the group of modular units $\mathcal{F}_{N,H}$ for $X_H(N)$ and the lattice $\mathcal{L}_{N,H}$.

\item[2.] Pick two linearly independent vectors $v_{1},v_{2}\in \mathcal{L}_{N,H}$ and compute the degrees $d_1, d_2$ and the inverses $X(x,y) :=\varphi^{-1}(v_{1})$ and $Y(x,y) := \varphi^{-1}(v_{2})\in \mathcal{F}_{N,H}$.

\item[3.]  Set $x_0:=2$ and $M:=\emptyset$.
\begin{itemize}
\item[3.1.]Let $\alpha$ be a root of $f_N(x_0,y)=0$, extract $[\mathbb{Q}(\alpha):\mathbb{Q}]$ linear equations from
$$\sum_{i\le d_2,~j\le d_1}c_{i,j}\cdot X(x_0,\alpha)^i\cdot Y(x_0,\alpha)^j=0.$$

\item[3.2.]Update the set of linear equations $M\leftarrow M\cup\{\textrm{new linear equations}\}$.

\item[3.3.]If  the solution space of $M$ has dimension greater than 1, then set $x_0\leftarrow x_0+1$ and go to step 3.1. Otherwise continue with step 3.4.

\item[3.4.]Pick a non-zero solution $(c_{i,j})$ and compute the genus of:
$$f_H(x,y):=\sum_{i\le d_2,~j\le d_1}c_{i,j}\cdot x^i\cdot y^j.$$
If this matches the genus of $X_H(N)$, then return $f_H(x,y)$, otherwise ``no output''.
\end{itemize}
\end{itemize}
\end{alg}
\begin{remark}
The algorithm is not guaranteed to always find an equation. Step~2 can only succeed if the
rank of $\mathcal L_{N,H}$ is at least 2. But even if it is, the two modular units corresponding to $v_1$ and $v_2$
need not generate $\mathbb Q(X_H(N))$. To check this, we compute the
genus in step 3.4, to ensure that if the algorithm returns an equation, then it will be correct.
\end{remark}
\begin{remark}Let $\mathcal{L}'_{N,H}$ be the lattice constructed from $\mathcal{L}_{N,H}$ by multiplying each entry with the degree of the corresponding cusp. The first two elements of an LLL basis of $\mathcal{L}'_{N,H}$ often
give modular units of $X_H(N)$ with optimal\footnote{Some $X_H(N)$ have functions of lower degree, but not modular units.} degrees.
In all cases we tried, their
polynomial relation had small coefficients as well; see Section~4 Table~\ref{equationforXHell} for
the equations ($f_H$ is denoted by $f_{N,[G:H]}$) of the $X_H(N)$ listed in Section~1 Table~\ref{comparingdims}.
\end{remark}

\subsection{The moduli interpretation for the plane model of $X_H(N)$}
Over an algebraically closed field $k$,
non-cuspidal points on $X_H(N)$ correspond to pairs $(E,S)$ where $E/k$ is an elliptic curve and $S$ is an $H$ orbit of points of
order $N$ on $E$. To apply Hecke operators on $J_H(N)$ we need this correspondence explicitly.
If we know the $x,y$ coordinates of a point $s$ on $X_1(N)$ then we also know its $b,c$ coordinates,
and the curve $E_{b(s),c(s)}$ with the point $(0,0)$ as in equation \ref{Tatenormalform}
will be the moduli interpretation of $s$. Let $\Pi: X_1(N) \to X_H(N)$ be the quotient map,
we can obtain the moduli interpretation of a point $s$ on $X_H(N)$ by computing a point $s' \in \Pi^{-1}(s)$
and then taking $(E,S)$ to be $(E_{b(s'),c(s')}, H(0,0))$. So we need to compute inverse images under $\Pi$.

Let $X(x,y)=\prod_{2\le i\le n+1}f_i(x,y)^{e_i}$ and $Y(x,y)=\prod_{2\le i\le n+1}f_i(x,y)^{g_i}$ then we can find
a polynomial relation between
the two elements in each pair $(X,x)$, $(X,y)$, $(Y,x)$ and $(Y,y)$ using resultants.
For example, let $N(x,y)$ and $D(x,y)$ be the numerator and denominator of $X(x,y)$ respectively, the polynomial relation (denoted by $P_{Xx}$) between $X$ and $x$ can be determined by computing
the resultant of the multivariate polynomials $D(x,y)\cdot T-N(x,y)$ and $f_N(x,y)$ with respect to the variable $y$
(for large $\ell$, we used a combination of evaluation and interpolation, like Algorithm~\ref{equtaionforXH}).  
These polynomials $P_{Xx},P_{Xy},P_{Yx}$ and $P_{Yy}$ will help us compute inverse images under $\Pi$.

\section{Constructing the representation space $V_{k,\ell}$ }
For each $k\in\{12,16,18,20,22,26\}$, let $\Delta_k=\sum_{n\ge1}\tau_k(n)q^n$ be the unique newform in $S_k(\SL_2(\mathbb{Z}))$ and $\Delta_{k,\ell}$ the newform in $S_2(\Gamma_1(\ell))$ which is congruent to $\Delta_k$ modulo $\ell$. Let $\ell>k$ be a prime number.
Associated to $\Delta_k$ there is a mod-$\ell$ Galois representation
$$\rho_{{k},\ell}:\GQ\to\textrm{Aut}(V_{k,\ell})\cong\GL_2(\mathbb{F}_\ell),$$
where $V_{k,\ell}=J_H(\ell)[\mathfrak{m}_{k,\ell}]$ ia a two-dimensional $\mathbb{F}_\ell$-vector space. More precisely, we have $V_{k,\ell}=A_{\Delta_{k,\ell}}[\mathfrak{m}_{k,\ell}]$. We can first construct points on $A_{\Delta_{k,\ell}}$ as follows. Let $\prod_f' A_f$ be a product of Abelian varieties, where $f$ runs through a set of representatives of $\GQ$-orbits of newforms in $S_2(\Gamma_1(\ell))$ with Dirichlet characters trivial on $H$, excluding the orbit of $\Delta_{k,\ell}$. For any Hecke operator $T_n$, denote by $\phi_n(x)$ the characteristic polynomial of $T_n$ on $\prod_f' A_f$, then we have a map \begin{equation}\label{phi_n}
\phi_n:J_H(\ell)\to A_{\Delta_{k,\ell}},~P\mapsto \phi_n(T_n)(P).
\end{equation}
Similarly, given $\ell$-torsion points on $A_{\Delta_{k,\ell}}$, we can construct $\ell$-torsion points on $V_{k,\ell}$ as follows. Let $\mathcal{S}$
be a set of positive integers, satisfying that $\ell$ and the $T_n-\tau_k(n)$, $n\in\mathcal{S}$ generate $\mathfrak{m}_{k,\ell}$. For every positive integer $n\in\mathcal{S}$, let $B_n(x)$ be the characteristic polynomial of $T_n$ on $A_{\Delta_{k,\ell}}$, then $B_n(x)$ can be factored as
\begin{equation}
B_n(x)=A_n(x)(x-\tau_k(n))\mod\ell.
\end{equation}
Define a composite map as
\begin{equation}\label{pi_S}
\pi_\mathcal{S}:A_{\Delta_{k,\ell}}[\ell]\to A_{\Delta_{k,\ell}}[\ell], ~P\mapsto \left(\prod_{n\in\mathcal{S}}A_n(T_n)\right)(P).
\end{equation}
Then for each point $P\in A_{\Delta_{k,\ell}}[\ell]$, $\pi_\mathcal{S}(P)$ is annihilated by $T_n-\tau_k(n)$ for all $n\in\mathcal{S}$. In other words, we have a map
$$\pi_\mathcal{S}:A_{\Delta_{k,\ell}}[\ell]\to V_{k,\ell}.$$

It seems hard to compute nonzero $\ell$-torsion points in $A_{\Delta_{k,\ell}}[\ell](\overline{\mathbb{Q}})$ directly.
It is easier to find points in $A_{\Delta_{k,\ell}}[\ell](\overline{\mathbb{F}}_p)$. So to construct
$V_{k,\ell}$, we first compute $V_{k,\ell}\mod p$ for sufficiently many
small prime numbers $p$ and then reconstruct $V_{k,\ell}$ with the Chinese Remainder Theorem.

Let $p$ be a prime number.
Elements of the two dimensional $\mathbb{F}_\ell$-vector space $V_{k,\ell}\mod p$ can be constructed as follows. Let 
$\mathbb{F}_q$ be a finite extension of $\mathbb{F}_p$ for which
$V_{k,\ell}(\overline{\mathbb{F}}_p)=V_{k,\ell}(\mathbb{F}_q)$. We have a map
\begin{equation}\label{pi_k_ell}
\pi_{k,\ell}:J_H(\ell)(\mathbb{F}_q)\xrightarrow{\phi_n} A_{\Delta_{k,\ell}}(\mathbb{F}_q)\xrightarrow{\psi} A_{\Delta_{k,\ell}}(\mathbb{F}_q)[\ell]\xrightarrow{\pi_{\mathcal{S}}}V_{k,\ell}\mod p,
\end{equation}
where $\phi_n$ is the map from (\ref{phi_n})
with some positive integer $n \ge 2$ (usually we take $n=2$, see \cite{ZengYin} for more)
and $\pi_{\mathcal{S}}$ is the map from (\ref{pi_S}). Define:
\begin{equation} \label{psi}
\psi:A_{\Delta_{k,\ell}}(\mathbb{F}_q)\rightarrow A_{\Delta_{k,\ell}}(\mathbb{F}_q)[\ell],~P\mapsto N_P \cdot P
\end{equation}
where $N_P$ is a divisor of $|A_{\Delta_{k,\ell}}(\mathbb{F}_q)| / \ell$
with minimal $\ell$-valuation for which $\ell \cdot N_P \cdot P$ vanishes.

Computing the map $\pi_{k,\ell}$ now comes down to computing the action of Hecke operators on $J_H(\ell)(\mathbb{F}_q)$,
as follows. Let $O$ be a $\mathbb{Q}$-rational cusp of $X_H(\ell)$, which serves
as the origin of the Jacobi map. The reduction modulo $p$ of $O$ is an
$\mathbb{F}_p$-rational point of $X_H(\ell)_{\mathbb{F}_p}$, denoted by $O$ as well.
Every point of $J_H(\ell)(\mathbb{F}_q)$ is represented as $P:=\sum_{i=1}^d P_i-gO$, where each $P_i$ is a place of $X_H(\ell)_{\mathbb{F}_q}$ and $g=\dim J_H(\ell)$.
Computing the action
of a Hecke operator $T_n$ on $P$ splits into three parts:
(1) compute $R_1:=\sum_{i=1}^d T_n(P_i)$,
(2) compute $R_2:=T_n(O)$,
(3) represent $R_1-gR_2$ as $Q:=\sum_{i=1}^{h}Q_i-gO$. Here we only explain part(1) in detail, as
part(2) can be found in \cite{ZengYin} and part(3) is realized with He\ss's algorithm \cite{Hess}.

Notations as in Section 2, let $f_\ell(x,y)$ be a defining equation for $X_1(\ell)$, and
$f_H(X,Y)$ a defining equation for  $X_H(\ell)$. Using these plane models, we have a map $\Pi:X_1(\ell)\to X_H(\ell)$,
$(x,y)\mapsto (X(x,y),Y(x,y))$. Given a point $(X_0,Y_0)$ with $f_H(X_0,Y_0)=0$, there are $h:=|H/\{\pm1\}|$ points above $(X_0,Y_0)$, which are denoted as $(x_{i},y_{i})$, $1\le i\le h$. We obtain $(x_i,y_i)$, $1\le i\le h$ by solving the equations
\begin{equation}
f_\ell(x,y)=0,~ X(x,y)=X_0, ~Y(x,y)=Y_0,
\end{equation}
and
\begin{equation}
P_{Xx}(X_0,x)=0,~ P_{Xy}(X_0,y)=0, ~P_{Yx}(Y_0,x)=0, ~P_{Yy}(Y_0,y)=0.
\end{equation}
The action of $T_n$ on $(X_0,Y_0)$ is computed as follows. Let $(x_0,y_0)$ be any of the $(x_{i},y_{i}), 1\le i\le h$
and $(b_0,c_0)$ its $(b,c)$ coordinates.
Let $n$ a prime number,
then $T_n(x_0,y_0)$ can be computed by the formula
$$T_n(E_{b_0,c_0},(0,0))=\sum_{C}(E_{b_0,c_0}/C,(0,0)+C)$$
where the sum is taken over all order $n$ subgroup $C \subset E_{b,c}$ such that $C \cap\langle(0,0)\rangle=\{0_{E_{b,c}}\}$.
Let $S$ be the set of points occuring in the summation above, then:
$$T_n(X_0,Y_0)=\sum_{s\in S}\Pi(s) \in J_H(\ell).$$

So using the map $\Pi:X_1(\ell)\to X_H(\ell)$ and He\ss's algorithm, we can compute $\pi_{k,\ell}(P)\in V_{k,\ell}(\mathbb{F}_q)$ for every $P\in J_H(\mathbb{F}_q)$ explicitly. Since $\dim_{\mathbb{F}_\ell}V_{k,\ell}=2$, a basis of $V_{k,\ell}$ can be found without difficulty.

We now explain how to determine the minimal extension field $\mathbb{F}_q=\mathbb{F}_{p^{d_p}}$ such that $V_{k,\ell}(\overline{\mathbb{F}}_p)=V_{k,\ell}(\mathbb{F}_q)$. The characteristic polynomial of the Frobenius endomorphism $\Frob_p$ on $V_{k,\ell}$ is
$$X^2-\tau_k(p)X+p^{k-1}\in \mathbb{F}_\ell[X].$$
So we have
$$d_p\le \min\{d\ge1:X^d=1 \textrm{ in } \mathbb{F}_\ell[X]/(X^2-\tau_k(p)X+p^{k-1})\}$$
and the equality holds if $\tau_k(p)^2-4p^{k-1}\not=0\in \mathbb{F}_\ell$. In practice, we would like to chose those prime numbers with small extension degree $d_p$ (more precisely, with small value $p^{d_p}$).
The exact value of $\tau_k$ at small (e.g. $p<10^7$) primes $p$ can be computed using the Fourier expansion of Eisenstein series,
$$\Delta_{12}=\frac{E_4^3-E_6^2}{1728} \textrm{ and } \Delta_{k}=E_{k-12}\cdot\Delta_{12} \textrm{ for } k\in\{16,18,20,22,26\}.$$
For even $k\ge4$, the weight $k$ level one Eisenstein series is defined as
$$E_k=1-\frac{2k}{B_k}\sum_{n=1}^\infty\sigma_{k-1}(n)q^n$$
where $\sigma_{k-1}(n)$ is the sum of the $(k-1)$th powers of the positive divisors of $n$, and $B_k$ is the $k$th Bernoulli number.

\begin{example}The case $k=16$ and $\ell=29$. As described in Section 1, the representative space of the mod-29 representation associated to $\Delta_{16}$ is $V_{16,29}=J_H(29)[\mathfrak{m}_{16,29}]$, where $H=\langle2^2\rangle\subset(\mathbb{Z}/29\mathbb{Z})^\times$. The Jacobian variety $J_H(29)$ has dimension 4 with an isogenous decomposition $J_H(29)\sim A_1\times A_2$, where $\dim A_1=2$ and $\dim A_2=2$. The $q$-expansion of newforms  associated to $A_1$ and $A_2$ are

$f_1=q + \alpha q^2 - \alpha q^3 + (-2\alpha - 1)q^4 - q^5 + (2\alpha - 1)q^6 + (2\alpha + 2)q^7 + (\alpha - 2)q^8  + \textrm{O}(q^{9})$,

$f_2=q + \beta q^2 - \beta q^3 - 3q^4 - 3q^5 + 5q^6 + 2q^7 - \beta q^8 - 2q^9 - 3\beta q^{10} + \beta q^{11} + \textrm{O}(q^{12})$\\ respectively. Here $\alpha$ is a root of $x^2+2x-1$ and $\beta$ is a root of $x^2+5$. The reductions  of $f_2$ modulo 29 are

$f_{2,1}=q + 16q^2 + 13q^3 + 26q^4 + 26q^5 + 5q^6 + 2q^7 + 13q^8 + 27q^9 + 10q^{10} + 16q^{11} + \textrm{O}(q^{12})$,

$f_{2,2}=q + 13q^2 + 16q^3 + 26q^4 + 26q^5 + 5q^6 + 2q^7 + 16q^8 + 27q^9 + 19q^{10} + 13q^{11} + \textrm{O}(q^{12}).$\\
We can check that $\Delta_{16}\mod 29=f_{2,2}$ and $\mathfrak{m}_{16,29}=\langle 29,T_2-\tau_{16}(2)\rangle$. Let $p=18443$ be a prime number, then $d_p=2$, $V_{16,29}(\overline{\mathbb{F}}_p)=V_{16,29}({\mathbb{F}_{p^2}})$, $|A_2({\mathbb{F}_{p^2}})|= 2^4\cdot3^6\cdot5^2\cdot7^2\cdot29^4\cdot107^2$ and the map $\pi_{k,\ell}$ (\ref{pi_k_ell}) is
$$\pi_{16,29}:J_H(29)(\mathbb{F}_{p^2})\xrightarrow{T_2^2+2T_2-1}A_2(\mathbb{F}_{p^2})
\xrightarrow{\times 404460 }A_2(\mathbb{F}_{p^2})[29]\xrightarrow{T_2+13}V_{16,29}({\mathbb{F}_{p^2}}).$$
To get a 29-torsion point in $A_2({\mathbb{F}_{p^2}})$ we first multiply a point in $A_2({\mathbb{F}_{p^2}})$ by
$404460 = 2^2\cdot3^3\cdot5\cdot7\cdot107$ ($\approx \sqrt{|A_2({\mathbb{F}_{p^2}})|}$).
The result has $29^*$-torsion. To get a $29$-torsion point, multiply by a suitable power of 29.
This approach worked well in all cases.
\end{example}

Instead of using a low degree function of $\mathbb{Q}(X_H(\ell))$ (see \cite{ZengYin}), we follow the method proposed by Mascot \cite{Mascot} to construct a function $\iota:V_{k,\ell}(\overline{\mathbb{Q}})\to \overline{\mathbb{Q}}$ such that $\sigma(\iota(x))=\iota(\sigma(x))$ for any $x\in V_{k,\ell} - \{0\}
$ and $\sigma\in \GQ$.
Mascot's method significantly reduces the coefficient sizes of:
$$P_{k,\ell}(X):=\prod_{x\in V_{k,\ell} - \{0\}}(X-\iota(x))\in\mathbb{Q}[X]$$
and
$$Q_{k,\ell}(X):=\prod_{L\in\mathbb{P}( V_{k,\ell})}(X-\sum_{x\in L - \{0\}}\iota(x))\in\mathbb{Q}[X].$$

\noindent \begin{longtable}{|c|c|c|c|c|}
\caption{\#digits of largest coefficient of $P_{k,\ell}(X)$ and $Q_{k,\ell}(X)$, original/Mascot}\label{comparingsizes} \\
 \hline
    $(k,\ell)$     &(12,13)  & (12,17)  &  (12,19) & (12,31)  \\
 \hline
   $P_{k,\ell}(X)$  &69/41      & 367/168       &  685/228      & ?/815 \\
 \hline
   $Q_{k,\ell}(X)$  &24/17       &215/72         &407/100         &1275/336   \\
 \hline
\end{longtable}

For each nonzero point $x\in V_{k,\ell}(\overline{\mathbb{Q}})$, let $D_x$ be the reduction of $x$ along the origin $O$, i.e. $D_x=x+\theta(x)O$, where $\theta(x)$ is the smallest integer such that $x+\theta(x)O$ is effective linearly equivalent to an effective divisor.  Then the Riemann-Roch space $\mathcal{L}(D_x):=\{f\in\overline{\mathbb{Q}}(X_H(\ell))^\times:\textrm{div}(f)+D_x\ge0\}\cup\{0\}$ has dimension one. Let $f_x$ be a nonzero element of  $\mathcal{L}(D_x)$. Fix two $\mathbb{Q}$-rational cusps of $X_H(\ell)$, denote as $O_1$ and $O_2$. We further assume that $\{O_1,O_2\}\cap\textrm{supp}(f_x)=\varnothing$ for all nonzero point $x\in V_{k,\ell}(\overline{\mathbb{Q}})$. Now we have a well-defined function
$$\iota:V_{k,\ell} - \{0\} \to \overline{\mathbb{Q}},~x\mapsto \frac{f_x(O_1)}{f_x(O_2)}$$
satisfying $\sigma(\iota(x))=\iota(\sigma(x))$, $\forall x\in V_{k,\ell} - \{0\}$, $\sigma\in\GQ$.

Let $p$ be a $\mathfrak{m}_{k,\ell}$-good prime (see \cite{Bruin}). Then for any nonzero point $\tilde{x}\in V_{k,\ell}(\overline{\mathbb{F}}_p)$, we have $\tilde{D}_x=\tilde{x}+\theta(\tilde{x})\tilde{O}$, where $\tilde{D}_x$, $\tilde{x}$ and $\tilde{O}$ are the reduction modulo $p$ of ${D}_x$, ${x}$ and ${O}$, respectively. Similarly, $\mathcal{L}(\tilde{D}_x):=\{f\in\overline{\mathbb{F}}_p(X_H(\ell))^\times:\textrm{div}(f)+\tilde{D}_x\ge0\}\cup\{0\}$ has dimension one. Let  $\tilde{f}_x$ be a nonzero element of $\mathcal{L}(\tilde{D}_x)$, we have
$$\tilde{\iota}:V_{k,\ell}(\overline{\mathbb{F}}_p) - \{0\} \to \overline{\mathbb{F}}_p,~x\mapsto\frac{\tilde{f}_x(\tilde{O}_1)}{\tilde{f}_x(\tilde{O}_2)}$$
satisfying that $\sigma(\tilde{\iota}(x))=\tilde{\iota}(\sigma(x))$ for all $x\in V_{k,\ell}(\overline{\mathbb{F}}_p) - \{0\}$ and $\sigma\in\textrm{Gal}(\overline{\mathbb{F}}_p/\mathbb{F}_p)$. So we have
$$\tilde{P}_{k,\ell}(X):=\prod_{x\in V_{k,\ell}(\overline{\mathbb{F}}_p) - \{0\}}(X-\tilde{\iota}(x))\in\mathbb{F}_p[X]$$
and
$$\tilde{Q}_{k,\ell}(X):=\prod_{L\in\mathbb{P}( V_{k,\ell}(\overline{\mathbb{F}}_p))}(X-\sum_{x\in L - \{0\}}\tilde{\iota}(x))\in\mathbb{F}_p[X]$$
such that $P_{k,\ell}(X)\mod p=\tilde{P}_{k,\ell}(X)$ and $Q_{k,\ell}(X)\mod p=\tilde{Q}_{k,\ell}(X)$.

One of the key results of \cite{Edixhoven} is that
the heights of $P_{k,\ell}(X)$ and $Q_{k,\ell}(X)$ can be bounded by a polynomial in $\ell$.
%
We use Mascot's $\tilde\iota$ because it is a significant improvement (Table~\ref{comparingsizes}).
Without a suitable height bound for $\tilde\iota$, to prove that
the output of our algorithm is correct we simply compute $P_{k,\ell}(X)$ and $Q_{k,\ell}(X)$ modulo enough
primes until the reconstructed polynomials have the right properties, and then prove correctness afterwards.
Note that there might even be a more optimal choice of $\tilde\iota$ since the degree $\ell+1$ polynomial $Q_{k,\ell}(X)$ can be reduced (polredabs in PARI/GP)
to a polynomial of much smaller height that still defines the same number field, see Table \ref{Polysforprojectiverepns}.

Let $\tilde{\rho}_{k,\ell}:\GQ\to\textrm{GL}_2(\mathbb{F}_\ell)\to\textrm{PGL}_2(\mathbb{F}_\ell)$ be the projective representation associated to $\Delta_{k}\mod\ell$. Then $\tilde{\rho}_{k,\ell}$ factors through $\textrm{Gal}(K_{k,\ell}/\mathbb{Q})$, where $K_{k,\ell}$ is the fixed field of $\ker \tilde{\rho}_{k,\ell}$, i.e. 
the splitting field of $Q_{k,\ell}^{red}(X)$. The following theorem 
helps to determine the matrix $\tilde{\rho}_{k,\ell}(\textrm{Frob}_p)$.

\begin{theorem}\label{Dokchitser's method} (Theorem 1.1. in \cite{Dokchitser})
Let $K$ be a global field and $f(x)\in K[x]$ a separable polynomial with Galois group $G$ and roots $a_1,\ldots,a_n$ in some splitting field. There is a polynomial $h(x)\in K[x]$ and polynomials $\Gamma_C\in K[X]$ indexed by the conjugacy classes $C$ of $G$, defined as
$$\Gamma_C(X)=\prod_{\sigma\in C}(X-\sum_{j=1}^n h(a_j)\sigma(a_j))$$
such that
$$\textrm{Frob}_\wp\in C\iff \Gamma_C(\textrm{Tr}_{\frac{\mathbb{F}_q[x]}{f(x)}/\mathbb{F}_q}(h(x)x^q))=0\mod\wp$$
for almost all primes $\wp$ of $K$, here $\mathbb{F}_q$ is the residue field at $\wp$.
\end{theorem}

The `almost all primes' in the theorem are those not dividing the denominators of the coefficients
of $f$, its leading coefficient and the resultants of $\Gamma_C(X)$ and $\Gamma_{C'}(X)$ for all $C\not= C'$.

Let
 $$I_{k,\ell}:\mathbb{Q}[X]/(Q_{k,\ell}(X))\to \mathbb{Q}[X]/(Q_{k,\ell}^{red}(X))$$
be the isomorphism between the field defined by $Q_{k,\ell}$ and its polredabs-reduced polynomial $Q_{k,\ell}^{red}$, then
$$\prod_{L\in\mathbb{P}( V_{k,\ell}(\overline{\mathbb{F}}_p))}(X-I_{k,\ell}(\sum_{x\in L - \{0\}}\tilde{\iota}(x)))=Q_{k,\ell}^{red}(X) \mod p.$$
Let
$\tilde{\alpha}_i$, $1\le i\le \ell+1$ be the roots
of $Q_{k,\ell}^{red}(X)\mod p$ for some splitting prime $p$.
We Hensel lift each root $\tilde{\alpha}_i$ to a root $\alpha_i \in\mathbb{Q}_p$
of $Q_{k,\ell}^{red}(X)$ with high $p$-adic precision. Then we can recover the polynomial
$$\Gamma_C(X):=\prod_{\sigma\in C}\left(X-\sum_{i=1}^{\ell+1}h(\alpha_i)\sigma(\alpha_i) \right)\in\mathbb{Q}[X]$$
where $h(x)\in \mathbb{Q}[x]$ is a small
auxiliary polynomial and $C$ is a conjugacy class of $\textrm{PGL}_2(\mathbb{F}_\ell)$. In practice, taking $h(x)=x^2$ sufficed each time
to identify the Frobenius endomorphism. Given $\Gamma_C(X)$ for all conjugacy classes $C$ of $\textrm{PGL}_2(\mathbb{F}_\ell)$ and a prime number $p$ not dividing $\textrm{Res}(\Gamma_C(X),\Gamma_{C'}(X))$ for all $C\not=C'$, then we have
$$\tilde{\rho}_{k,\ell}(\textrm{Frob}_p)\in C\iff\Gamma_C(\textrm{Tr}_{\frac{\mathbb{F}_p[x]}{Q_{k,\ell}^{red}(x)}/\mathbb{F}_p}(h(x)x^p))\equiv0\mod p.$$
Since $\textrm{Tr}({\rho}_{k,\ell}(\textrm{Frob}_p))\equiv \tau_k(p)\mod\ell$ and $\textrm{Det}({\rho}_{k,\ell}(\textrm{Frob}_p))\equiv p^{k-1}\mod\ell$, we have for all $\lambda$ in $\mathbb F_\ell^\times$ that
\begin{equation}\label{taupmodl}
\tau_k(p)\equiv\pm \sqrt{\frac{p^{k-1}}{\textrm{Det}(\lambda{\rho}_{k,\ell}(\textrm{Frob}_p))}}\cdot\textrm{Tr}(\lambda{\rho}_{k,\ell}(\textrm{Frob}_p))\mod\ell.
\end{equation}
So we can reconstruct $\tau_k(p)$ up to sign knowing only $\tilde{\rho}_{k,\ell}(\textrm{Frob}_p)$.
The steps of computing the polynomials $P_{k,\ell}$ and $Q_{k,\ell}$  corresponding to the representation space $V_{k,\ell}$ are summarized as follows.
\begin{alg}\label{algorithm2}Constructing the representation space.

Input: Level $\ell$, weight $k$, $f_H(x,y)$, $f_\ell(x,y)$ and a set of positive integers $\mathcal{S}$ such that $\mathfrak{m}_{k,\ell}=\langle\ell,T_n-\tau_k(n):n\in\mathcal{S}\rangle$.

Output: $P_{k,\ell}$ and $Q_{k,\ell}$.
\begin{itemize}
\item[1.]  Initialization.
\begin{itemize}
\item[1.1.] Set $p:=\ell+1$, $M:=\emptyset$ and $O, O_1, O_2$ three distinct rational cusps of $X_H(\ell)$.

\item[1.2.] Compute the characteristic polynomial $\phi_2(x)$ of $T_2$ on $\prod_f' A_f$.

\item[1.3.] Compute the characteristic polynomial of $T_n$ on $A_{\Delta_{k,\ell}}$ for all $n\in\mathcal{S}$.
\end{itemize}
\item[2.] Search the next prime number $p$ such that $d_p\le 4$, and then set $q=p^{d_p}$.

\item[3.] Compute $|J_H(\ell)(\mathbb{F}_q)|$, $|A_{\Delta_{k,\ell}}(\mathbb{F}_q)|$ using modular symbol algorithms for modular forms.

\item[4.] Pick a random point $P$ on $J_H(\ell)(\mathbb{F}_q)$ and compute the image $\pi_{k,\ell}(P)$ using the function field $\mathbb{F}_q(x)[y]/(f_H(x,y))$ and He\ss's algorithm.
\item[5.] Find $e_1$ and $e_2$ such that $V_{k,\ell}=\mathbb{F}_\ell e_1+\mathbb{F}_\ell e_2$ by running Step 4 for several times.

\item[6.] For $0\le i,j\le \ell-1$ and $(i,j)\not=(0,0)$, compute the reduction of each point $x_{ij}:=ie_1+je_2$ along $O$, denote it as $D_{ij}$. If the stability $\theta(x_{ij})$ is not equal to the genus of $X_H(\ell)$ then goto Step 2, otherwise compute the one-dimensional Riemann-Roch space $\mathcal{L}(D_{ij})$ and
let $f_{ij}$ be a basis of $\mathcal{L}(D_{ij})$. Let $q_{i,j} := f_{ij}(O_1)/f_{ij}(O_2)$.

\item[7.] Compute the polynomials $\tilde{P}_{k,\ell}$ and $\tilde{Q}_{k,\ell}$ as
$$\tilde{P}_{k,\ell}:=\prod_{0\le i,j\le \ell-1,(i,j)\not=(0,0)}\left(X-q_{i,j}\right)$$
$$\tilde{Q}_{k,\ell}:=\left(X-\sum_{i=1}^{\ell-1} q_{i,0} \right)\cdot\left(X-\sum_{j=1}^{\ell-1} q_{0,j} \right)\cdot \prod_{i=1}^{\ell-1}\left(X-\sum_{j=1}^{\ell-1}
q_{ij, j}
\right)$$
where the $ij$ in $q_{ij,j}$ refers to $ij$ mod $\ell$.
Update $M:=M\cup\{(\tilde{P}_{k,\ell},p),(\tilde{Q}_{k,\ell},p)\}$.
\item[8.] Try to reconstruct ${P}_{k,\ell}$ and ${Q}_{k,\ell}$ over $\mathbb{Q}$ from $M$ using the
Chinese Remainder Theorem. If this succeeds, test the polynomials and output them, otherwise goto Step 2.
\end{itemize}
\end{alg}

\section{Implementations and computational results}

The algorithms in Sections 2 and 3 are implemented using Magma \cite{Bosma}. We found particularly simple equations for $X_H(\ell)$ and largely extend the computational results on Galois representations associated to level one modular forms as well.

We follow the strategy proposed by Bosman (see Chapter 7 of \cite{Edixhoven}) to prove Theorem \ref{mainthm} in Section~1.
The proof is divided into two parts. First, we verified using Magma that each polynomial $Q^{red}_{k,\ell}(x)$ in Table \ref{Polysforprojectiverepns} has the right Galois group, i.e. $\textrm{Gal}(Q^{red}_{k,\ell}(x))\cong\textrm{PGL}_2(\mathbb{F}_\ell)$.

Then, we verify that the Galois representation $\GQ\to\textrm{PGL}_2(\mathbb{F}_\ell)$ arising from the isomorphism $\textrm{Gal}(Q^{red}_{k,\ell}(x))\cong\textrm{PGL}_2(\mathbb{F}_\ell)$ has the right Serre invariants (level and weight). Using Tate's theorem on lifting projective Galois representations (see \cite{Serre} for the proof) and a theorem of Moon and Taguchi \cite{Moon-Taguchi}, we verify that the lifted representation $\rho:\GQ\to\textrm{GL}_2(\mathbb{F}_\ell)$ has weight $k$ and level 1 by verifying that the discriminant of the number field $\mathbb{Q}[x]/(Q^{red}_{k,\ell}(x))$ is $(-1)^{(\ell-1)/2}\ell^{k+\ell-2}$. The verification procedure includes a program for computing the discriminant of number fields with small ramified
primes, and can be found at the website mentioned in Section 1.

The proof of Corollary \ref{maincor} in Section 1 is based on the following lemma and a fact pointed out by Serre \cite{Serre2}, i.e. if $p$ is prime number satisfying $\tau(p)=0$ then $p$ can be written as
$$p=hM-1$$
with
$$M=2^{14}3^75^3691,$$
$$\left(\frac{h+1}{23}\right)=1, \textrm{and } h\MOD 49 \in\{ 0,30,48\}.$$
\begin{lemma}Let $Q^{red}_{k,\ell}(x)$ be a polynomial in Table \ref{Polysforprojectiverepns} and $p\nmid\textrm{Disc}(Q^{red}_{k,\ell}(x))$ a prime number. Then
$\tau_k(p)\equiv 0 \mod\ell$ if and only if $Q^{red}_{k,\ell}(x)\MOD p$ has an irreducible factor of degree 2 over $\mathbb{F}_p$.
\end{lemma}
\begin{proof}See \cite{Edixhoven} Chapter 7, Lemma 7.4.1.
\end{proof}

Using the polynomials defining the projective representations associated to $\Delta_{12}\mod\ell$ for $\ell \in\{11,13,17,19,29,31,41\}$, we found the first prime number $p$ satisfying Serre's criteria as well as $\tau(p)\equiv0\mod11\cdot13\cdot17\cdot19\cdot29\cdot31\cdot41$ is $816212624008487344127999$. So Corollary \ref{maincor} is proved.

Using  Dokchitser's method (Theorem \ref{Dokchitser's method}) for finding Frobenius element and formula \ref{taupmodl},
the value of  $\tau(p)\mod41$ for a large prime $p$ can be computed efficiently, for example
$$\tau(10^{1000}+1357)\equiv\pm8\mod41.$$

%

\renewcommand\arraystretch{1.5}
\tabcolsep=5.5pt
\begin{longtable}{|>{\RaggedRight}p{10mm}<{\centering}   |>{\RaggedRight}p{150mm} <{\centering}|}
\caption{Equation for $X_H(\ell)$}\label{equationforXHell}\\
  \hline
  $X$&$\scriptstyle{f_{3}^ {3}f_{5}^ {3}f_{6}f_{8}^ {6}f_{13}f_{15}^ {9}}/(f_{4}^ {5}f_{7}^ {7}f_{9}^ {3}f_{10}f_{11}f_{12}f_{14}^ {10})$\tabularnewline \hline
  $Y$&$\scriptstyle f_{2}^ {7}f_{3}^ {62}f_{5}^ {66}f_{6}^ {6}f_{8}^ {168}f_{15}^ {210}/(f_{4}^ {102}f_{7}^ {198}f_{9}^ {54}f_{10}^ {6}f_{11}^ {6}f_{12}^ {6}f_{13}^ {6}f_{14}^ {216})$\tabularnewline \hline
  $f_{29,2}$&$\scriptstyle-(X^2-5X-1)^7Y^3+X^3(4X^2+9X-4)(34X^6+621X^5+389X^4+917X^3-389X^2+621X-34)Y^2+X^5(92X^3+25X^4+25-92X+25X^2)Y+X^7$\tabularnewline\hline
  $X$&$\scriptstyle 1/f_{10}$\tabularnewline \hline
  $Y$&$\scriptstyle f_{5}f_{8}f_{15}/(f_{4}f_{7}^ {2}f_{14})$\tabularnewline \hline
  $f_{29,7}$& $\scriptstyle X^2Y^6+(2X-1)(X-1)Y^5+(X-1)(2X^2-X+1)Y^4+2X^2(X-1)^2Y^3+X^2(3X-2)(X-1)Y^2+X(2X^2-2X+1)(X-1)^2Y+X^4(X-1)^2$ \tabularnewline \hline
  $X$&$\scriptstyle f_{5}^ {3}f_{8}^ {2}f_{9}f_{10}f_{11}f_{13}f_{15}^ {2}/(f_{4}^ {4}f_{6}f_{7}^ {4}f_{12}f_{14}f_{16}^ {2})$\tabularnewline \hline
  $Y$&$\scriptstyle f_{4}^ {7}f_{8}f_{9}^ {3}f_{16}^ {4}/(f_{3}^ {2}f_{5}^ {3}f_{6}^ {5}f_{7}^ {2}f_{12}f_{14}f_{15}^ {3})$\tabularnewline \hline
  $f_{31,3}$&$\scriptstyle Y^6-3XY^5-X(-9X+2X^2-4)Y^4+X(X-1)(4X^2+3X+2)Y^3+X^2(X^4+9X^2-3X^3+9-X)Y^2-X^2(2X^2-4X+3)Y+X^2$\tabularnewline \hline
  $X$&$\scriptstyle f_{3}^ {10}f_{5}^ {20}f_{6}^ {27}f_{8}^ {5}f_{10}f_{11}f_{12}f_{14}f_{15}^ {27}/(f_{2}f_{4}^ {40}f_{7}f_{9}^ {20}f_{16}^ {27})$\tabularnewline \hline
  $Y$&$\scriptstyle f_{3}^ {14}f_{5}^ {28}f_{6}^ {36}f_{8}^ {9}f_{10}f_{11}f_{12}f_{13}f_{14}f_{15}^ {39}/(f_{2}f_{4}^ {56}f_{7}^ {3}f_{9}^ {27}f_{16}^ {38})$\tabularnewline \hline
  $f_{31,5}$&$\scriptstyle (-6X^2-X^4-4X^3-1-4X)Y^6+(-X^5+1+20X^2+14X+9X^3)Y^5+(-14X^2+4X^5+9X^4-11X+7X^3)Y^4+(5X^2-5X^5-13X^3-6X^4)Y^3+(20X^3+6X^5+7X^4)Y^2+(9X^4+X^6+2X^5)Y+X^5$\tabularnewline \hline
  $X$&$\scriptstyle f_{2}^ {3}f_{4}^ {44}f_{5}^ {12}f_{7}^ {16}f_{9}^ {70}f_{19}^ {114}/(f_{3}^ {42}f_{6}^ {104}f_{8}^ {34}f_{10}^ {2}f_{11}^ {2}f_{12}^ {2}f_{13}^ {2}f_{14}^ {2}f_{15}^ {2}f_{16}^ {2}f_{17}^ {2}f_{18}^ {116})$\tabularnewline \hline
  $Y$&$\scriptstyle f_{3}^ {5}f_{6}^ {14}f_{8}^ {7}f_{13}f_{14}f_{15}f_{17}f_{18}^ {16}/(f_{4}^ {8}f_{7}^ {4}f_{9}^ {9}f_{10}f_{11}f_{12}f_{16}f_{19}^ {15})$\tabularnewline\hline
  $f_{37,2}$&$\scriptstyle X^5Y^6+(36X^5+33X^4)Y^5+(-466X^4+429X^5+51X^3-13X^2)Y^4+(542X^4+1656X^5+63X^2-91X^3-5X+1)Y^3+(13X^2+466X^4-429X^5-51X^3)Y^2+(36X^5+33X^4)Y-X^5$\tabularnewline \hline
 $X$&$\scriptstyle f_{3}f_{8}f_{13}f_{18}^ {3}/(f_{5}f_{9}f_{10}f_{11}f_{14}f_{19}^ {2})$\tabularnewline \hline
 $Y$&$\scriptstyle f_{3}^ {4}f_{6}^ {10}f_{8}^ {3}f_{12}f_{15}f_{16}f_{17}f_{18}^ {10}/(f_{4}^ {4}f_{5}f_{7}^ {2}f_{9}^ {6}f_{13}f_{19}^ {11})$\tabularnewline \hline
  $f_{37,3}$&$\scriptstyle -X^5Y^5+(5X^4-9X^2-3X^3-5X-1)Y^4-X(9X^2-2X-3)Y^3+3X(X+3)Y^2+5XY+X
$\tabularnewline \hline
 $X$&$\scriptstyle f_{11}/f_{13}$\tabularnewline \hline
 $Y$&$\scriptstyle f_{7}f_{12}f_{19}/(f_{8}f_{13}f_{18})$\tabularnewline \hline
  $f_{37,9}$&$\scriptstyle -X^2Y^9+X^2(X+1)Y^8+X(X+1)(X^3+X^2+X-2)Y^7-(X+1)(X^5-X^4+2X^3-4X^2-6X-1)Y^6+(X^5-2X^4+2X^3-2X^2-4X-2)(X+1)^2Y^5+(X^7-2X^4-4X^3-6X^2-X+1)(X+1)^2Y^4-X(X^5-X^4-X^3-6X^2-8X-4)(X+1)^2Y^3-X(X^4-X^3-X+2)(X+1)^3Y^2+X^2(X^2-2X-2)(X+1)^3Y+X^2(X+1)^3$\tabularnewline \hline
 $X$&$\scriptstyle f_{4}^ {144}f_{8}^ {282}f_{10}^ {216}f_{11}^ {3}f_{12}^ {3}f_{13}^ {3}f_{14}^ {3}f_{15}^ {3}f_{16}^ {3}f_{17}^ {3}f_{18}^ {3}f_{19}^ {3}f_{20}^ {213}/(f_{2}^ {5}f_{3}^ {64}f_{6}^ {24}f_{7}^ {246}f_{9}^ {117}f_{21}^ {210})$\tabularnewline \hline
 $Y$&$\scriptstyle f_{3}^ {2}f_{6}^ {2}f_{7}^ {8}f_{9}^ {3}f_{16}f_{18}f_{21}^ {8}/(f_{4}^ {4}f_{8}^ {10}f_{10}^ {6}f_{11}f_{12}f_{13}f_{14}f_{15}f_{17}f_{19}f_{20}^ {7})$\tabularnewline \hline
 $f_{41,2}$&$\scriptstyle X^8Y^5-Y^4(58Y^2-58-5Y)X^6+1777Y^4(1+Y^2)X^5+Y^3(1816Y^4+1816-2169Y-49863Y^2+2169Y^3)X^4+
Y^2(1+Y^2)(29Y^4-23832Y^3+77627Y^2+23832Y+29)X^3-Y^2(-881536Y^5+13620Y^6-9728957Y^3-881536Y+3748542Y^2-
3748542Y^4-13620)X^2+Y(1+Y^2)(280Y^6+148329Y^5+1274508Y^4-73664834Y^3-1274508Y^2+148329Y-280)X-(Y^2-64Y-1)^5
$\tabularnewline \hline
 $X$&$\scriptstyle f_{3}f_{6}f_{7}^ {5}f_{21}^ {3}/(f_{4}^ {2}f_{8}^ {4}f_{10}^ {3}f_{16}f_{18}f_{20}^ {2})$\tabularnewline \hline
 $Y$&$\scriptstyle f_{4}^ {8}f_{5}^ {2}f_{8}^ {20}f_{10}^ {15}f_{14}f_{15}f_{17}f_{19}f_{20}^ {15}/(f_{3}^ {5}f_{6}f_{7}^ {19}f_{9}^ {7}f_{11}f_{12}f_{13}f_{21}^ {15})$\tabularnewline \hline
 $f_{41,4}$&$\scriptstyle X^5Y^6+(5X^5-7X^4-X^3-5X^2-5X-1)Y^5+(-18X^4+5X^5-7X-18X^2+6X^3)Y^4+(6X^4+X-6X^2+X^5)Y^3+(-18X^4-18X^2+7X^5-6X^3-5X)Y^2+
(X^6-5X^5-X^3+5X^4+5X+7X^2)Y-X$\tabularnewline \hline
$X$&$\scriptstyle f_{4}f_{8}^ {5}f_{10}^ {3}f_{13}f_{18}f_{20}^ {3}/(f_{3}f_{6}f_{7}^ {5}f_{12}f_{15}f_{21}^ {3})$\tabularnewline \hline
$Y$&$\scriptstyle f_{3}^ {2}f_{6}f_{7}^ {9}f_{9}^ {3}f_{14}f_{21}^ {6}/(f_{4}^ {4}f_{8}^ {10}f_{10}^ {6}f_{13}f_{18}f_{20}^ {6})$\tabularnewline \hline
$f_{41,5}$&$\scriptstyle Y^8X^8-2Y^4(3Y^3+Y-1-2Y^2)X^7+(28Y^5+Y^2-Y-21Y^6+1-2Y^3+8Y^7-13Y^4)X^6+(-52Y^5+49Y^4+5Y-3+16Y^6-17Y^3)X^5+
(48Y^5-Y^2+36Y^3-7Y^6-70Y^4+3-7Y)X^4+(Y-1)(5Y^5-28Y^4+29Y^3-6Y^2-2Y+1)X^3+Y^2(-5-25Y^2+20Y+11Y^3)X^2-3Y^2(Y-1)^3X+Y^2(Y-1)^3$\tabularnewline \hline				    	
$X$&$\scriptstyle f_{3}^ {2}f_{5}^ {2}f_{6}^ {5}f_{7}^ {6}f_{10}^ {2}f_{12}f_{19}f_{20}f_{21}^ {3}/(f_{4}^ {4}f_{8}^ {6}f_{9}^ {3}f_{11}^ {5}f_{16}f_{22}^ {4})$\tabularnewline \hline
$Y$&$\scriptstyle f_{3}^ {2}f_{6}^ {4}f_{7}^ {12}f_{13}f_{14}f_{15}f_{17}f_{18}f_{21}^ {6}/(f_{4}^ {2}f_{8}^ {10}f_{9}^ {2}f_{10}^ {2}f_{11}^ {6}f_{12}f_{19}f_{20}f_{22}^ {6})$\tabularnewline \hline
$f_{43,3}$&$\scriptstyle X^4Y^{10}-3X^4(2X+5)Y^9+X^4(48X^2+92+111X+2X^3)Y^8-X^3(15X^5+526X^2+128X^4+15+313X^3+329X)Y^7+
X^3(15X^6+X^7+92X^5+128+1499X^2+329X^4+1011X+1011X^3)Y^6-X^2(1499X^2+1499X^3+526X^4+313X+2+111X^5+6X^6)Y^5+
X^2(313X^3+48X^4+1011X^2+526X+48)Y^4-X^2(329X+111+128X^2+2X^3)Y^3+X(15X^2+6+92X)Y^2-15XY+1		    	
$\tabularnewline \hline
\end{longtable}

\newpage

\renewcommand\arraystretch{1.5}
\tabcolsep=5.5pt
\begin{longtable}{|>{\RaggedRight}p{20mm}<{\centering}|>{\RaggedRight}p{140mm}<{\centering}|}
\caption{Polynomials for projective representations.}\label{Polysforprojectiverepns}\\
 \hline
    $(k,\ell)$     & $Q_{k,\ell}^{red}(x)$  \\
 \hline
 $(12,29)$ &$x^{30}-3x^{29}-87x^{26}-348x^{25}+3364x^{24}-3016x^{23}-7627x^{22}-11078x^{21}-18792x^{20}+
426967x^{19}-912630x^{18}-2141853x^{17}+13020884x^{16}-20860106x^{15}-13673152x^{14}+
130529116x^{13}-211945746x^{12}-34076769x^{11}+639742407x^{10}-1393480566x^9+
2123886688x^8-2443924830x^7+2204756987x^6-1665273496x^5+908690959x^4-
253988728x^3+11893480x^2+1767126x+1745477
$ \tabularnewline \hline
   $(12,41)$  &$x^{42}-10x^{41}+82x^{40}-533x^{39}+3075x^{38}-16400x^{37}+80975x^{36}-370640x^{35}+
1519378x^{34}-5564971x^{33}+17920690x^{32}-49655756x^{31}+115329720x^{30}-
206406669x^{29}+240009203x^{28}+131055598x^{27}-1265809933x^{26}+2505951898x^{25}-
1541118824x^{24}-7804079523x^{23}+33765940074x^{22}-62075585470x^{21}+8325418672x^{20}
+199473849030x^{19}-310656709925x^{18}-104883756359x^{17}+753871386098x^{16}-
566628777936x^{15}-586603185734x^{14}+1044692298155x^{13}-229863404090x^{12}-
525813316148x^{11}+848849997762x^{10}-393039966527x^{9}-558892835247x^{8}
+872468938446x^{7}+244708057774x^{6}-517918323061x^{5}+44959617394x^{4}+
211599644868x^{3}-38057695352x^{2}-39907608565x+20280586312$\tabularnewline \hline
 $(16,29)$&$x^{30}-13x^{29}+116x^{28}-899x^{27}+6003x^{26}-33002x^{25}+142158x^{24}-437871x^{23}+599981x^{22}+3161522x^{21}
-30157709x^{20}+149069425x^{19}-545068137x^{18}+1602112888x^{17}-3929042061x^{16}+8240756348x^{15}-15020495335x^{14}
+23992472995x^{13}-33394267804x^{12}+40034881756x^{11}-40888329774x^{10}+35730188833x^9-27316581262x^8+17713731976x^7
-7068248851x^6-1463296732x^5+4054490087x^4-2555610007x^3+2573924261x^2+2363203645x-261910751$\tabularnewline \hline
 $(16,43)$&$x^{44}-2x^{43}+129x^{41}-903x^{40}-946x^{39}+14491x^{38}-111499x^{37}+92622x^{36}+1691319x^{35}-7697860x^{34}+13044050x^{33}+112312388x^{32}
-431011016x^{31}+635764116x^{30}+4479038627x^{29}-17986900688x^{28}+17761002123x^{27}+113734065567x^{26}-499847503435x^{25}+
286497523429x^{24}+2354056386953x^{23}-10459820630338x^{22}+12147259637525x^{21}+20067741453659x^{20}-130000420363335x^{19}+
384688694954926x^{18}-485820321658115x^{17}+105790129474267x^{16}+2848140899460525x^{15}-4719664885329376x^{14}+5652559275048886x^{13}
+5788815326549276x^{12}-11470729218321031x^{11}+59608728270470171x^{10}-100591321161751298x^9+180921441441713877x^8-
173177824071705704x^7+281127869880019503x^6-226711920334690606x^5+196387509281481601x^4-80193983753034520x^3+
80212380437798668x^2+18055666311398579x+26917038595190660$\tabularnewline \hline
 $(18,29)$&$x^{30}-x^{29}+29x^{28}-116x^{27}+435x^{26}-3248x^{25}+9947x^{24}-57652x^{23}+348145x^{22}-
1668022x^{21}+9627797x^{20}-45146359x^{19}+193617601x^{18}-785410306x^{17}+
2590336115x^{16}-8081143103x^{15}+21726617822x^{14}-48427338815x^{13}+92038539344x^{12}
-101938632221x^{11}-23635971065x^{10}+420966210322x^9-1674835015037x^8+3364868756248x^7
-2315012281648x^6-5570532030513x^5+12983251570145x^4-1986227939226x^3
-8282713720550x^2-6298098956973x-1759183949485$\tabularnewline \hline
$(18,37)$&$x^{38}-19x^{37}+222x^{36}-1776x^{35}+10915x^{34}-52281x^{33}+192178x^{32}-449587x^{31}-44289x^{30}+
7347904x^{29}-48388600x^{28}+217625675x^{27}-771830767x^{26}+
2358345442x^{25}-5373731075x^{24}+7137737784x^{23}+21109958392x^{22}-
202047172726x^{21}+946901905837x^{20}-3348744435690x^{19}+9462953254801x^{18}-
22880275166766x^{17}+47253834997237x^{16}-88981319473988x^{15}+
157905318032031x^{14}-264902704553019x^{13}+405055047903218x^{12}-
472340175239772x^{11}+227866385419064x^{10}+
406381259592545x^9-1335445860560463x^8+1772915458368853x^7-790671316155658x^6-1184027867126053x^5+1742488440361558x^4-719886497838982x^3-
366627422362591x^2+423416835435221x-141000935841284$\tabularnewline \hline
$(18,41)$&$x^{42}-2x^{41}-1394x^{38}+6642x^{37}-30176x^{36}+42025x^{35}+26281x^{34}-595320x^{33}+7529609x^{32}-2464592x^{31}+71028031x^{30}+268813630x^{29}+
1886054243x^{28}+5435754088x^{27}+25817655679x^{26}+69418136849x^{25}+405207224722x^{24}+869997292888x^{23}+3805858723156x^{22}+7443690885348x^{21}+
32586137482517x^{20}+87689775068921x^{19}+258614112547199x^{18}+392541861700062x^{17}+656103042884690x^{16}+914306104529073x^{15}+3223926858581507x^{14}+
5231616358024282x^{13}+8838554621239255x^{12}+1914012057065804x^{11}+8299844168687279x^{10}+20100370223826819x^9+87436795827343019x^8+
63759507173044916x^7+123906821068999365x^6+246831919541533446x^5+442638825282242467x^4+545141334744587773x^3+472194537897724139x^2+
221738918882225639x+51726007574271052
$\tabularnewline \hline
$(20,31)$&$x^{32}-4x^{31}-62x^{30}+558x^{29}-248x^{28}-23560x^{27}+143499x^{26}+59489x^{25}-
4280108x^{24}+17190864x^{23}+12517459x^{22}-344750256x^{21}+1225662500x^{20}-
278789479x^{19}-14790203106x^{18}+64357190741x^{17}-83774789980x^{16}-
406418167694x^{15}+2480836111912x^{14}-5273524311353x^{13}-3257558862543x^{12}+
54285321863574x^{11}-162450534558477x^{10}+197719989210108x^9+250865100757790x^8-
1714511602191278x^7+4206562171750919x^6-6661579151098950x^5+7460752526582377x^4-
5959749341609879x^3+3269911760551427x^2-1113936554991727x+178725601175511$\tabularnewline \hline
$(20,37)$&$x^{38}-8x^{37}-74x^{35}+740x^{34}+592x^{33}+1998x^{32}-31524x^{31}-
53502x^{30}-46842x^{29}+861952x^{28}+2186034x^{27}+1913344x^{26}-
16066584x^{25}-53860900x^{24}-64902440x^{23}+185452399x^{22}+
861406916x^{21}+1402064606x^{20}-873188826x^{19}-8790442758x^{18}-
18594523974x^{17}-8518042764x^{16}+48955601616x^{15}+145848475197x^{14}+
169649084430x^{13}-42316100208x^{12}-556412529242x^{11}-1063909014419x^{10}-
1054270670966x^9-39590377918x^8+1840697590018x^7+3934087735017x^6+5222200302936x^5+
5163607007328x^4+4108937512400x^3+2432911826752x^2+1111769639264x+373259137536$\tabularnewline \hline
$(22,29)$&$x^{30}-15x^{29}+145x^{28}-957x^{27}+5423x^{26}-28362x^{25}+190704x^{24}-1215738x^{23}+
9030890x^{22}-50755684x^{21}+302445002x^{20}-1541053533x^{19}+8227183130x^{18}-
36312207320x^{17}+118566870446x^{16}-105292107135x^{15}-1691924793888x^{14}
+13329033978106x^{13}-58590901828799x^{12}+169683493444005x^{11}-315306359440118x^{10}
+273465061439983x^9+46526649574358x^8+1224234551141538x^7-10361450966128074x^6+
34659815824570960x^5-67880128879605199x^4+80477031276514924x^3-43420672886426530x^2-
8341058879724619x-1870319891085836$\tabularnewline \hline
$(22,31)$&$x^{32}-3x^{31}-124x^{30}+651x^{29}+5797x^{28}-44020x^{27}-46593x^{26}+1523309x^{25}-
4960682x^{24}-28562129x^{23}+205283395x^{22}+345367838x^{21}-3865963779x^{20}-
5281917640x^{19}+35629245810x^{18}+95827452774x^{17}+227525150938x^{16}-
1735983387875x^{15}-9952753525850x^{14}+15867354189588x^{13}+
146446287180279x^{12}-99789981007214x^{11}-1135328992145553x^{10}-
171825071648506x^9+7446294546204081x^8+294530833190147x^7-
24397472702475140x^6-9976638213111902x^5+61714590456038129x^4+
16902762581347117x^3-13833080015551423x^2-202960986205176103x
+187532019539254309$\tabularnewline \hline
$(22,37)$&$x^{38}-14x^{37}+259x^{36}-1961x^{35}+21016x^{34}-109446x^{33}+921263x^{32}-3307023x^{31}+25225157x^{30}-54040794x^{29}+468439277x^{28}-
177732571x^{27}+6697779183x^{26}+15102386533x^{25}+95823354911x^{24}+467027027922x^{23}+1552177493664x^{22}+8186043757227x^{21}+
23326397273054x^{20}+102773425686308x^{19}+284599390053820x^{18}+1002115462118745x^{17}+2706484715050559x^{16}+7589293500429809x^{15}+
18210715535225813x^{14}+38743465024834893x^{13}+68669684492163440x^{12}+87971508991005731x^{11}+39575428039113138x^{10}-
152455852846579113x^9-629771290390520323x^8-1341552601370955658x^7-2412402149619660642x^6-3173510643427726774x^5-
3412849621807994643x^4-2880329122123064182x^3-2037223141300178414x^2-1317582443043720513x-434573469219991349$\tabularnewline \hline
$(22,41)$&$x^{42}-18x^{41}+123x^{40}-861x^{39}+7380x^{38}-40959x^{37}+231978x^{36}-1413352x^{35}+
6996486x^{34}-38081907x^{33}+189053706x^{32}-917457451x^{31}+4315693702x^{30}-18315504830x^{29}+
81234872473x^{28}-318790621656x^{27}+1367441283261x^{26}-5493803836853x^{25}+23948461399133x^{24}-
99508282718538x^{23}+428746575253775x^{22}-1734584463212166x^{21}+
6876735749430348x^{20}-24827905121788763x^{19}+84833216659986525x^{18}-259437778230162837x^{17}+
738066411091063506x^{16}-1849044869912543534x^{15}+4213586181863965064x^{14}-
8156005605599678501x^{13}+13798349887560500249x^{12}-18028595022441906291x^{11}+
18399612027049480612x^{10}-6833934943203545127x^9-8147747464527030097x^8+
30429728300791774349x^7-24476859699695334748x^6+23454931837408314608x^5+20127918822514029803x^4-
5069638438935010043x^3+30654410371510538992x^2+
33409106662514832514x+15131308571734590785$\tabularnewline \hline
$(26,29)$&$x^{30}-8x^{29}-406x^{27}+3161x^{26}+56405x^{25}-548564x^{24}-790250x^{23}+10634532x^{22}+140239940x^{21}-1035077541x^{20}-1328813611x^{19}+18748981719x^{18}+
82518484681x^{17}-1026164148328x^{16}+4120059344955x^{15}-6818998593950x^{14}-45412134549190x^{13}+87588273401126x^{12}+1920846833452511x^{11}-
7102239213427725x^{10}-7241201654936390x^9+38056121342434911x^8+107005415665914120x^7-433431488131317544x^6+93915016050246679x^5+
24302411958402756x^4+2152162187366627404x^3-2627576944604601474x^2-601832151572568240x-3098422653204866539$\tabularnewline \hline
$(26,31)$&$x^{32}-4x^{31}-31x^{30}+868x^{29}-5642x^{28}+26939x^{27}-
128681x^{26}+878261x^{25}-5466230x^{24}+18081277x^{23}+
20896666x^{22}-590595663x^{21}+3237658414x^{20}-
7823924026x^{19}-8271478466x^{18}+144440084360x^{17}-
545248581427x^{16}+814731722043x^{15}+1545029759212x^{14}-
10100559817781x^{13}+13746840290876x^{12}+50937263692756x^{11}-
312389013580541x^{10}+794269279837735x^9-1279246263404303x^8+
1890280924355725x^7-4573026243585585x^6+12387587893505272x^5-
15936935475740110x^4-5793602599713315x^3+57908374101008694x^2-
86833942062161928x+43493578237258823
$\tabularnewline \hline
$(26,37)$&$x^{38}-74x^{35}+814x^{34}+1184x^{33}-5513x^{32}+1406x^{31}+75739x^{30}+
187812x^{29}+2708326x^{28}+6715574x^27-11574895x^{26}-19400580x^{25}-
60540954x^{24}-257031489x^{23}+219930960x^22+476178678x^{21}-7725949206x^{20}+
7817760303x^{19}+90069526203x^{18}+188854397226x^{17}+190346904225x^{16}-
1554089198490x^{15}-7584658041190x^{14}-13677702715890x^{13}-12884820525090x^{12}-
6813726842980x^{11}-11834535466411x^{10}-46685971384319x^9-52104622636843x^8-
8834390171891x^7-81330026874835x^6-141458791372569x^5-30179350468786x^4-
87042767581262x^3-137469175925255x^2+22045406053866x-42545911481352
$\tabularnewline \hline
$(26,41)$&$x^{42}-16x^{41}+205x^{40}-2296x^{39}+21853x^{38}-183639x^{37}+1408432x^{36}-9250297x^{35}+50067109x^{34}-222765177x^{33}+764959263x^{32}-764168045x^{31}-
17046189684x^{30}+196027654300x^{29}-1602496205751x^{28}+11584076519536x^{27}-73333374570606x^{26}+405985748093720x^{25}-2004709742983360x^{24}+
8715892574008402x^{23}-31747846180895478x^{22}+91431223779418992x^{21}-166629887328691848x^{20}-125100079850044093x^{19}+2621218007838892023x^{18}-
12217279090918678854x^{17}+26330915710315752864x^{16}+46134343749852064266x^{15}-785257752575733924122x^{14}+4784050476754172698578x^{13}-
20269417207395728781750x^{12}+70523538060019897017689x^{11}-216395855160190180594775x^{10}+590985612388177055013665x^9-1535533982339410498146959x^8
+3907575027035666938917328x^7-9391013229994259435004044x^6+20099094569742517895015811x^5-36634154316180650165044084x^4+
51478054781718097858971669x^3-60934951211352211286266257x^2+44645295386777494419181009x-29743028377517568874028591 $\tabularnewline \hline
\end{longtable}

\section{acknowledgments}
We would like to thank Marco Streng and Ronald van Luijk for letting us use about 50 cores of Capella for an extended period of time. This made it
possible to compute
many primes $p$ simultaneously, which greatly reduced the computation time.

\end{document}